\documentclass[11pt]{amsart}

\usepackage{amssymb,amsthm}
\usepackage{amsmath}

\newtheorem{theorem}{Theorem}[section]

\newtheorem{lemma}[theorem]{Lemma}

\def\irr#1{{\rm  Irr}(#1)}

\def\phi{{\varphi}}

\begin{document}

\title[Camina pairs]{Camina pairs that are not $p$-closed}

\author{Mark L. Lewis}

\address{Department of Mathematical Sciences, Kent State University, Kent, OH 44242}

\email{lewis@math.kent.edu}



\begin{abstract}
We show for every prime $p$ that there exists a Camina pair $(G,N)$ where $N$ is a $p$-group and $G$ is not $p$-closed.  
\end{abstract}


\maketitle

\section{Introduction}

Our purpose in this note is to construct a number of examples.  Throughout this paper, all groups will be finite.  In this note, we focus on Camina pairs.  If $N$ is a normal subgroup of $G$, we say that $(G,N)$ is a {\it Camina pair} if every element $g \in G \setminus N$ is conjugate to all of $gN$.  This is one of a number of equivalent formulations of this condition.

Camina has proved (see Theorem 2 of \cite{camina}) that if $(G,N)$ is a Camina pair, then either $G$ is a Frobenius group with Frobenius kernel $N$ or at least one of $N$ or $G/N$ is a $p$-group for some prime $p$.  In particular, a Camina pair $(G,N)$ is a Frobenius group with Frobenius kernel $N$ if and only if $|G:N|$ and $|N|$ are relatively prime (see Proposition 1 of \cite{camina}).  The study of Frobenius groups has its own long history, so the study of Camina pairs has focused on Camina pairs that are not Frobenius groups.  These Camina pairs can be divided into three categories: (1) $G$ is a $p$-group, (2) $G/N$ is a $p$-group and $N$ is not a $p$-group but $p$ divides $|N|$, and (3) $N$ is a $p$-group and $G/N$ is not a $p$-group but $p$ divides $|G:N|$.

In this note, we are particularly interested in Camina pairs in category (3).  In particular, $N$ is a $p$-group for some prime $p$ and $G/N$ is not a $p$-group, but $p$ does divide $|G/N|$.  Let $P$ be a Sylow $p$-subgroup of $G$.  A group $G$ is said to be $p$-closed if $P$ is normal in $G$.  Nearly all of the known examples of Camina pairs of this type are $p$-closed (see \cite{ChMc} and\cite{ChMaSc}). In particular, the only published non $p$-closed examples are two solvable groups in \cite{gagola}: one with $p = 2$ and one with $p = 3$ and one nonsolvable example in \cite{nonsolv} where $p = 5$.  On page 274 of \cite{ChMaSc}, it is suggested that the only Camina pairs $(G,N)$ where $N$ is a $p$-group and not $p$-closed have $p = 2$ and $p = 3$.   However, we will show that there exist solvable, non $p$-closed examples for every prime $p$.

\begin{theorem} \label{three}
Let $p$ be any prime, then there exists a Camina pair $(G,N)$ where $N$ is a $p$-group, and $G$ is solvable and not $p$-closed.  
\end{theorem}

In \cite{sylow}, it is shown that if $(G,N)$ is a Camina pair where $N$ is a $p$-group and $G$ is solvable, then $G$ has $p$-length at most $2$.  This shows that $p$-length equal to $2$ can occur for every prime $p$.

\section{Construction} \label{sec 2}

Let $F$ be a field of order $p^n$ where $p$ is a prime and $n$ is a positive integer.  A Sylow $p$-subgroup of ${\rm GL}_3 (F)$ can be represented with the following elements $\{ (a,b,c) \mid a,b,c \in F \}$ where the multiplication is as follows $(a_1,b_1,c_1) (a_2,b_2,c_2) = (a_1 + a_2, b_1 + b_2, c_1 + c_2 + a_1 \cdot b_2)$.  Sometimes this group is known as the Heisenberg group for $F$, so we will denote it by ${\mathcal H} (F) = {\mathcal H} (n,p)$.  (It is also seen to be the $3 \times 3$ upper triangular group with diagonal entries $1$.)  A quick computation shows that $[(a_1,b_1,c_1),(a_2,b_2,c_2)] = (0,0,a_1 b_2 - a_2 b_1)$.  From this, it is not difficult to see that $Z = \{ (0, 0, c) \mid c \in F \}$ is both the derived subgroup and the center of ${\mathcal H} (F)$.  Another computation shows that if $ab \ne 0$, then $(a,b,c)$ is conjugate to all of $(a,b,c)Z$ in ${\mathcal H} (F)$.  Hence, $(\mathcal {H}(F),Z)$ is a Camina pair.

As seen in \cite{gagola} and \cite{isasny}, there is a natural action of the multiplicative group $F^*$ of $F$ on ${\mathcal H} (F)$ that can be found in ${\rm GL}_3 (F)$.  Abstractly, this action can be written as $(a,b,c) \cdot x = (a,bx,cx)$ for all $(a,b,c) \in {\mathcal H} (F)$ and $x \in F^*$.  We write $\mathcal {K} (F) = \mathcal {K}(n,p)$ for the resulting semi-direct product of $F^*$ acting on $\mathcal {H} (F)$.  Consider the subgroup $B = \{ (0,b,c) \mid b,c \in F \}$ in $\mathcal {H} (F)$, and observe that $B F^*$ is a Frobenius group.  Since $Z \le B$, we see that $T = Z F^*$ is a Frobenius group with Frobenius kernel $Z$.  Since $\mathcal {K} (F) = \mathcal {H} (F) T$,  we may use Lemma 2.2 of \cite{ChMc} to see that $(\mathcal {K}(F), Z)$ is a Camina pair.  On the other hand, we see that $C_{\mathcal {H}(F)} (F^*) = \{ (a,0,0) \mid a \in F \}$, so $\mathcal {K} (F)$ is not a Frobenius group.  One can show that $[\mathcal {H}(F),F^*] = \{ (0,b,c) \mid b,c \in F \} = B$.  The key fact for building our examples is that ${\bf O}^p (\mathcal {K} (F)) = BF^*$ has index $p^n$ in $\mathcal {K}(F)$.

Take $S$ to be the Galois group for $F/Z_p$.  There is a natural action of $S$ on $\mathcal {K}(F)$ as follows: if $\sigma \in S$, then $(a,b,c) \cdot \sigma = (a^\sigma,b^\sigma,c^\sigma)$ for all $(a,b,c) \in P$ and $x \cdot \sigma = x^\sigma$ for all $x \in F^*$.  We write $\mathcal {G}(F) = \mathcal {G} (n,p)$ to be the resulting semi-direct product of $S$ acting on $\mathcal {K} (F)$.

\section{Theorem \ref{three}}

In this section, we prove Theorem \ref{three}.

It is useful to work in the situation studied by Gagola in \cite{gagola}.  Gagola studied groups that had an irreducible character $\chi$ so that $\chi$ vanishes on all but two conjugacy classes.  We will say that $\chi$ is a Gagola character.  It is shown that if $G$ has a Gagola character, then $G$ has a unique minimal normal subgroup $N$ and $N$ is elementary abelian $p$-group for some prime $p$ (see Lemma 2.1 of \cite{gagola}).  Using Theorem 2.5 (a) of \cite{gagola} and Proposition 3.1 of \cite{ChMc}, we see that $(G,N)$ is a Camina pair.  In fact, it is shown on page 383 of \cite{gagola} that $\mathcal K (n,p)$ has a Gagola character when $p$ is a prime and $n$ is a positive integer so that $p^n > 2$.

The work to proving Theorem \ref{three} is contained in the following lemma.  It would be interesting to see under what conditions this lemma can be generalized to Camina pairs.

\begin{lemma} \label{work}
Suppose $p$ is a prime.  Let $G$ be a group that has a subgroup $K$ of index $p$ so that $|K| > 2$, $K$ has a Gagola character $\theta$, and ${\bf O}_p (G) > 1$ is the Sylow $p$-subgroup of $K$.  Suppose that $M$ is normal in $G$ so that $M \le K$ and $|K:M| = p$.  If $H$ is a subgroup of $G$ so that $M < H < G$, $H \ne K$, and $H$ does not stabilize an irreducible constituent of $\theta_M$, then $H$ has a Gagola character and $H$ is not $p$-closed.
\end{lemma}

\begin{proof}
Let $N$ be the unique minimal normal subgroup of $K$.  By its uniqueness, it follows that $N$ is characteristic in $K$, so $N$ is normal in $G$.  We know that $\theta_N$ has all the nonprincipal characters in $\irr N$ as constituents (see Lemma 2.1 of \cite{gagola}).  Since ${\bf O}_p (G) > 1$ is a normal subgroup of $K$, we must have that $N \le {\bf O}_p (G)$.  Let $L = M \cap {\bf O}_p (G)$, and observe that $K/L = {\bf O}_p (G)/L \times M/L$, so $K$ is not a Frobenius group.  Let $\lambda \in \irr N$ be a nonprincipal character, and we have seen that $\lambda$ is an irreducible constituent of $\theta_N$.  Then we know that ${\bf O}_p (G)$ is the stabilizer of $\lambda$ in $K$, and $\lambda$ is fully ramified with respect to ${\bf O}_p (G)/N$ (see Corollary 2.3 of \cite{gagola}).

It follows that $L$ is the stabilizer of $\lambda$ in $M$.  Write $\zeta$ for the character in $\irr {{\bf O}_P (G) \mid \theta}$ so that $\zeta^K = \theta$.  Let $\eta$ be an irreducible constituent of $\zeta_L$.  Notice, that if $\eta = \zeta_L$, then $\eta (1)/\lambda (1) = \theta (1) = |{\bf O}_p (G):N|^{1/2} > |L:N|^{1/2}$ which contradicts Corollary 2.30 of \cite{text}.  Thus, $\zeta (1) = p \eta (1)$, and we see that $\irr {L \mid \lambda}$ contains $p$ distinct characters.  Let $\mu = \eta^M \in \irr {M \mid \lambda}$.  We see that $\mu^K = (\eta^M)^K = \eta^K = (\eta^{{\bf O}_p (G)})^K = (\zeta)^K = \theta$.  Hence, $\mu$ is an irreducible constituent of $\theta_M$.  It follows that $\mu$ is not $H$-invariant.  Since $|H:M| = p$, we see that $\gamma = \mu^H \in \irr H$ (see Corollary 6.19 in \cite{text}).

We now prove that $\gamma$ is a Gagola character for $H$.  By Lemma 2.1 of \cite{gagola}, $\theta$ is the unique Gagola character in $K$, and so, it is $G$-invariant.  As $|G:K| = p$, it follows that $\theta$ extends to $G$ and every irreducible constituent of $\theta^G$ is an extension of $\theta$.  Let $\chi$ be an irreducible constituent of $\gamma^G$.  We see that $\chi$ is a constituent of $\gamma^G = (\mu^H)^G = (\mu^K)^G = \theta^G$.  This implies that $\chi$ is an extension of $\theta$.  Thus, $\chi (1) = \theta (1) = |K:M| \mu (1) = p\mu (1) = |H:M| \mu (1) = \gamma (1)$.  We deduce that $\chi_H = \gamma$.  This implies that $\gamma_N = \chi_N$.  Thus, every nonprincipal character in $\irr N$ is a constituent of $\gamma_N$, and so, $H$ acts transitively on $\irr N \setminus \{ 1 \}$.  By Theorem 6.32 of \cite{text}, $H$ acts transitively on $N \setminus \{ 1 \}$.  We see that $|H| = |K| = |K|_p (|N|-1)$.  It follows that the stabilizer of $\lambda$ in $H$ is a Sylow $p$-subgroup $P$ of $H$.  Let $\tau \in \irr {P \mid \lambda}$ be the Clifford correspondent for $\gamma$.  Then $(|N| - 1) \tau (1) = \gamma (1) = \chi (1) = \theta (1) = (|N| - 1) e$ where $e^2 = |{\bf O}_p (G):L|$.  Thus, $\gamma (1) = e$ and $e^2 = |{\bf O}_p (G):L| = p|L:N| = |P:N|$.  Since $\lambda$ is $P$-invariant, this implies that $\lambda$ is fully ramified with respect to $P/N$.  It is not difficult to see that this implies that $\gamma$ vanishes on $H \setminus N$, and since $N$ consists of two conjugacy classes in $H$, we conclude that $\gamma$ is a Gagola character for $H$.

Finally, we need to show that $H$ is not $p$-closed.  Since $|G:K| = |G:H| = p$ and $H \ne K$, we have $G = KH$.  We see that $H \cap K = M$.  Also, $M$ and ${\bf O}_p (G)$ have coprime indices in $K$, so $K = {\bf O}_p (G) M$, and this implies that $G = KH = ({\bf O}^p (G) M)H = {\bf O}^p (G)H$.  Notice that $P$ has index $p$ in $P {\bf O}_p (G)$, so ${\bf O}_p (G)$ normalizes $P$.  Thus, if $H$ also normalizes $P$, then $P$ would be normal in $G$ which is a contradiction since $P$ is not contained in ${\bf O}_p (G)$.  Therefore, $H$ is not $p$-closed.  This proves the lemma.
\end{proof}

The following theorem includes Theorem \ref{three}.

\begin{theorem}
Let $p$ be any prime, then there exists a Camina pair $(G,N)$ where $N$ is a $p$-group, and $G$ is solvable and not $p$-closed.  In particular, $G$ can be chosen to have a Gagola character.
\end{theorem}

\begin{proof}
Let $F$ be the field of order $p^p$.  Take $G = \mathcal {G} (F) = \mathcal {G} (p,p)$ as constructed in Section \ref{sec 2}.  Let $K = \mathcal {K} (F)$ be as defined there.  We have $|K| = p^{3p} (p^p - 1) > 2$.  As was observed above, $K$ has a Gagola character $\theta$.  Observe that $|G:K| = p$, and note that ${\bf O}_p (K) > 1$ is the Sylow $p$-subgroup of $K$.  Observe that $G/{\bf O}_p (K)$ is the semi-direct product of a group of order $p$ acting on a cyclic group of order $p^p - 1$ coming from the Galois group $S$ of $F/Z_p$ on the multiplicative group of $F$.  Thus, a Sylow subgroup of $G/{\bf O}_p (K)$ has order $p$ and is not normal.  This implies that ${\bf O}_p (G) = {\bf O}_p (K)$.

We noted in Section \ref{sec 2} that ${\bf O}^p (K)$ has index $p^p$ in $K$.  Since $|G:K| = p$, it follows that ${\bf O}^p (G) = {\bf O}^p (K)$.  Since $G/{\bf O}^p (K)$ is a $p$-group of order $p^{p+1}$, we can find $M$ normal in $G$ and contained in $K$ so that $|K:M| = p$.  Note that $S \le G$ has order $p$ and is not contained in $K$.  Thus, $MS/M$ will be a subgroup of $G/M$ of order $p$ that is different that $K/M$.  This implies that $G/M$ is not cyclic.  Since $|G:M| = p^2$, we see that $G/M$ is elementary abelian.

Let $\gamma$ be an irreducible constituent of $\theta_M$, and let $T$ be the stabilizer of $\gamma$ in $G$.  Notice that $T$ will be normal in $G$, so it will be the stabilizer of all of the irreducible constituents of $\theta_M$.  We know that $G/M$ has exactly $p + 1 \ge 3$ subgroups of order $p$.  Thus, we can find $H/M$ a subgroup of $G/M$ having order $p$ that is not $K/M$ or $T/M$.  I.e., $H$ has index $p$ and is not $K$ and does not stabilize any irreducible constituent of $\theta_M$.  We can now apply Lemma \ref{work} to see that $H$ is a Gagola group that is not $p$-closed.  This proves the theorem.
\end{proof}

\end{document}